\pdfoutput=1
\documentclass{compositio}
\usepackage[utf8]{inputenc}
\usepackage[polish,english]{babel}
\usepackage{amssymb}
\usepackage{amsmath}
\usepackage{mathrsfs}
\usepackage{yfonts}
\usepackage{mathtools}
\usepackage{tikz-cd}
\usepackage{xcolor}
\usepackage{graphicx}
\usepackage[parfill]{parskip}
\usepackage{hyperref}
\hypersetup{
 colorlinks = true, 
 urlcolor = black, 
 linkcolor = black, 
 citecolor = black 
}
\newtheorem{theorem}{Theorem}
\newtheorem*{theorem*}{Main Theorem}

\newtheorem{proposition}[theorem]{Proposition}
\newtheorem{lemma}[theorem]{Lemma}

\newtheorem{definition}[theorem]{Definition}

\numberwithin{theorem}{section}

\newcommand\rr{\rightarrow}
\newcommand{\PP}{\mathbb P}
\newcommand{\ZZ}{\mathbb{Z}}
\newcommand{\QQ}{\mathbb{Q}}

\newcommand{\mc}{\mathcal}

\newcommand{\wdt}{\widetilde}

\begin{document}
\title{Counterexample \\to the N\'eron--Ogg--Shafarevich criterion \\for Calabi--Yau threefolds}

\author{Tymoteusz Chmiel}
\email{tymoteusz.chmiel@uj.edu.pl}
\address{Faculty of Mathematics and Computer Science, Jagiellonian University, ul. \L ojasiewicza 6, 30-348 Krak\'ow, Poland}

\author{Marcin Oczko}
\email{marcin.oczko@im.uj.edu.pl}
\address{Faculty of Mathematics and Computer Science, Jagiellonian University, ul. \L ojasiewicza 6, 30-348 Krak\'ow, Poland}


\thanks{The first author was supported by the National Science Center grant no. 2023/49/N/ST1/04089. The second named author was partially supported by National Science Centre, Poland, grant No. 2020/39/B/ST1/03358.}

\begin{abstract}
For any prime $p>5$ we construct a Calabi-Yau threefold $X$ defined over a finite extension $K$ of $\mathbb{Q}_p$ such that every model of $X$ over $\operatorname{Spec}\mathcal{O}_K$ has singular special fiber, yet the Galois action on the $\ell$-adic cohomology group $H^3_{\acute{e}t}(X,\mathbb{Q}_\ell)$ is unramified for $\ell\neq p$ and crystalline for $\ell=p$. This provides a counterexample to the analogue of the N\'eron-Ogg-Shafarevich criterion in dimension three.
\end{abstract}

\maketitle 

\section{Introduction}

A \emph{Calabi-Yau manifold} is a smooth projective variety $X$ with trivial canonical bundle and vanishing cohomology groups $H^i(X,\mathcal{O}_X)$ for $0<i<\text{dim}\,X$. Usually one considers Calabi-Yau manifolds defined over $\mathbb{C}$. On the other hand, in this paper we are concerned with arithmetic properties of Calabi-Yau manifolds and the field of definition is either $p$-adic or finite.

One-dimensional Calabi-Yau manifolds are \emph{elliptic curves} and two-dimensional ones are called \emph{K3 surfaces}. This paper is concerned with \emph{Calabi-Yau threefolds} and a phenomenon that does not appear in lower dimensions: the \'etale cohomology fails to detect bad reduction. Thus the natural generalization of the N\'eron-Ogg-Shafarevich criterion does not hold in dimension $3$.

N\'eron-Ogg-Shafarevich criterion states that an elliptic curve $E$ over a $p$-adic field has good reduction if and only if the Galois representation $H^1_{\acute{e}t}(E,\mathbb{Q}_\ell)$ is unramified for $\ell\neq p$ and crystalline for $\ell=p$ (see \cite{Ogg}). A version of this criterion holds for $K3$ surfaces: $K3$ surface $S$ admitting a semi-stable model has potentially good reduction if and only if the Galois representation $H^2_{\acute{e}t}(S,\mathbb{Q}_\ell)$ is unramified for $\ell\neq p$ and crystalline for $\ell=p$ (see \cite{LM,BLL}).

It means that for Calabi-Yau varieties of dimension $\leq 2$ good reduction is detected by the Galois action on the middle cohomology. We show it is not the case in dimension three:

\begin{theorem*}
Let $p>5$ be a prime, $K:=\mathbb{Q}_p[\sqrt{p}]$ and denote by $G_K$ the absolute Galois group of $K$. There exists smooth Calabi-Yau threefold $Y_p$ defined over $K$ such that:
\begin{itemize}
 \item the Galois representation
 $$G_K\rightarrow \operatorname{Aut}\left( H^3_{\acute{e}t}(Y_p,\mathbb{Q}_\ell )\right)$$
 is unramified for $\ell \neq p$ and crystalline for $\ell =p$;
 \item $Y_p$ does not have a smooth model over $\operatorname{Spec}\mathcal{O}_L$ for any finite field extension $L/K$.
\end{itemize}
\end{theorem*}

The Galois action on the \'etale cohomology of a $p$-adic variety is an arithmetic analogue of the monodromy action on the rational cohomology of a complex manifold. In \cite{CvS} Cynk and van Straten constructed a family of complex Calabi-Yau threefolds over the punctured unit disc $\Delta^*$ such that the monodromy action of the fundamental group $\pi_1(\Delta^*,t)\simeq \mathbb{Z}$ on $H^3(X_{t},\mathbb{C})$ is trivial, yet the family cannot be completed to a smooth family over the entire unit disc $\Delta$.

To prove our Main Theorem we repurpose their example in mixed characteristic. In fact, the variety $Y_p$ is defined over $\mathbb{Q}[\sqrt{p}]\hookrightarrow\mathbb{C}$ and is a smooth member of the family of Cynk-van Straten. Thus our result proves the conjecture stated at the end of \cite{CvS}.

The structure of the paper is as follows. In section \ref{sec:construction} we present a detailed construction of the manifold $Y_p$. In section \ref{sec:crystalline} we prove that the Galois action on its cohomology is crystalline. In section \ref{sec:no-good} we show that $Y_p$ does not admit a smooth integral model over any finite extension of the base field $K$.

\section{Construction}\label{sec:construction}

Our example is constructed as a desingularization of a double cover of $\PP^3$ branched along a union of eight planes with no $k$-fold lines for $k \geq 4$ and no $l$-fold points for $l \geq 6$. Double covers of $\PP^3$ branched along such plane arrangements were introduced in \cite{CYB} and are called \emph{double octics}; their branching locus is called an \emph{octic arrangement}. Double octics provide a rich source of examples of Calabi-Yau threefolds:
\begin{theorem}[Theorem 2.1, \cite{CYB}]
\label{resolution of octic}
Let $D \subset \PP^3$ be an octic arrangement. There exists a sequence of blow-ups $\sigma = \sigma_1 \circ ... \circ \sigma_s \colon \wdt{\PP^3} \rr \PP^3$ and a smooth even divisor $D^* \subset \wdt{\PP^3}$ such that $\sigma_*(D^*) = D$ and the double covering of $\wdt{\PP^3}$ branched along $D^*$ is a smooth Calabi-Yau threefold.
\end{theorem}

We apply Theorem \ref{resolution of octic} to the following family of double octics:
\begin{equation}\label{double-octic}
X_t
:= \Big\{u^2 = xy (x + y) z (x + 2 y + z + tv) v (y + z + v) (x + y + z + (t - 1) v)\Big\}
\end{equation}
Here $t\in\PP^1$ is a parameter of the family and $(u:x:y:z:v) \in \PP(4,1,1,1,1)$ are coordinates of the weighted projective space containing double octics $X_t$. The construction is similar to that of \cite{CvS} in the complex case but we include it for the sake of completeness.

We label the eight planes defining $X_t$ as $P_1,...,P_8$, in the same order as their equations appear in (\ref{double-octic}). Note that we suppress the dependence on $t$ from notation. Each $X_t$ is a double cover of $\mathbb{P}^3$ branched along the octic arrangement $P_1\cup\cdots\cup P_8$.

For generic value of $t\ (\neq 0,1,2,\infty)$ double octic $X_t$ has the following singularities:
\begin{itemize}
\item a single triple line $l_3:=\{x=y=0\}$;
\item 25 double lines;
\item six fourfold points, none of them lying on the triple line $l_3$;
\item five fourfold points on the triple line $l_3$.
\end{itemize}
When $t=0$, two fourfold points $(0:0:0:1)$ and $(0:0:-t:1)$ coincide to create a fivefold point lying on the planes $P_1,P_2,P_3,P_4,P_5$. Other singularities remain the same.

Let $P\subset X_0$ be the plane spanned by $P_4 \cap P_5$ and $l_3$. Four planes $P_1,P_2,P_3,P$, all containing the line $l_3$, define four points in $\PP^1$ identified with the space of planes through $l_3$. Let $E$ be the double cover of $\PP^1$ branched at these four points. It is an elliptic curve with the $j$-invariant $j(E)=1728$. We will later see that its cohomology appears in the cohomology of the special fiber of the family (see Theorem \ref{th:coh}).

Fix a prime $p > 5$. Equations of $X_p$ have coefficients in $\ZZ\subset \mathbb{Z}_p$, hence they define a scheme $\mc{X}\rightarrow\operatorname{Spec}\mathbb{Z}_p$ whose generic fiber is $X_p$. The variety $X_p$ is aso a fiber of the family $X_t$, $t\in\mathbb{Z}$, and its reduction modulo $p$ is the same as reduction of $X_0$. We shall denote $X_\infty:=X_p$.

We construct a scheme $\mc{Y}$ whose generic fiber $Y_\infty$ is a resolution of $X_\infty$ (see Theorem \ref{resolution of octic}), while the central fiber $Y_0$ is a simple normal crossing divisor with rigid Calabi-Yau threefold as an irreducible component. A Calabi-Yau threefold $X$ is called \emph{rigid} if it has no deformations of the complex structure. By Bogomolov-Tian-Todorov theorem it is equivalent to $b_3(X)=2$.

We start resolving the generic fiber of $\mc{X}$ by the following sequence of blow-ups:
\begin{enumerate}
 \item Let $\sigma_1 \colon \PP^{(1)} \rr \PP^3$ be the blow-up in the triple line $l_3$ and in all fourfold points which do not lie on $l_3$. Let $\mc{D}^{(1)}$ be the strict transform of the branching locus $\mc{D}$ plus the exceptional divisor $H = \PP^1 \times \PP^1$ of the triple line. We define $\mc{X}^{(1)}$ as the double cover of $\PP^{(1)}$ branched along $\mc{D}^{(1)}$. The branch divisor of $\mc{X}^{(1)}$ contains 8 new double lines coming from intersections of $H$ with the 8 planes in the branching divisor; denote these lines as $m_i := P_i \cap H,i=1, ... ,8$. The lines $m_1,m_2,m_3$ are in one ruling of $H$, while the other 5 lines are in the second ruling. In local coordinates near the line $m_1$, the scheme $\mc{X}^{(1)}$ is given by equation
 \[
 u^2 = xyz(x+2xy+z+p)F(u,x,y,z,w)
 \]
 In this coordinates $x$ corresponds to the exceptional divisor, while $y,z$ and $(x+2xy+z+p)$ to strict transforms of $P_1,\ P_4$ and $P_5$ respectively. Factor $F$, with $F(\bold{0})\neq 0$, corresponds to the other surfaces in the branching divisor.

 \item Let $ \sigma_2 \colon \PP^{(2)} \rr \PP^{(1)}$ be the blow-up of $\PP^{(1)}$ in all double lines other than $m_4,\ m_5$ and $P_4 \cap P_5$. Let $\mc{D}^{(2)}$ be the strict transform of $\mc{D}^{(1)}$. Define $\mc{X}^{(2)}$ as the double cover of $\PP^{(2)}$ branched along $\mc{D}^{(2)}$. Over the line $m_1$, which in local coordinates is given by $x=y=0$, this blow-up can be described in two affine charts as
 \[
 u^2 = yz(x + 2x^2y+z+p) \quad
 u^2 = xz(xy+2xy^2 +z+p)
 \]

 \item Let $ \sigma_3 \colon \PP^{(3)} \rr \PP^{(2)}$ be the blow-up of $\PP^{(3)}$ in double curves $m_4,\ m_5$ and $P_4 \cap P_5$, which are the only remaining singularities of the generic fiber. Let $\mc{D}^{(3)}$ be the strict transform of $\mc{D}^{(2)}$ and let $\mc{X}^{(3)}$ be a double cover of $\PP^{(3)}$ branched along $\mc{D}^{(3)}$. Consider affine charts above $m_1$. In the first one, the branch divisor is a simple normal crossing with smooth components both in the generic, and in the degenerate fiber. Thus by blowing up intersections of the components of the branch divisor we will not introduce any singularities. This is not the case for the second chart, since the intersection $z = xy+2xy^2 + z +p =0$ is singular $\mod p$. In this chart the blow-up is in the ideal $\left(xz,x(xy+2xy^2 +z+p),z(xy+2xy^2 +z+p),u\right)$. To compute this blow-up we take the closure of the graph of the map
 
 \[
 (x,y,z,u) \rr (X,Y,Z,T) = (xz,x(xy+2xy^2 +z+p),z(xy+2xy^2 +z+p),u)
 \]
 By explicit computations we verify that the generic fiber $X^{(3)}_\infty$ is smooth, and the central fiber $X^{(3)}_0$ is singular along the line $L = \{x=z=u=X=Y=Z=0\}$, with four pinch points lying on this line.

 \item Let $\mc{X}^{(4)}$ be the blow-up of $\mc{X}^{(3)}$ in the ideal $(p,I(L))$. This blow-up transforms the degenerate fiber $X^{(3)}_0$ into a union of a rigid Calabi-Yau threefold $R_0$ and a $\mathbb{P}^2$-bundle $V_0\rightarrow L$, intersecting $R_0$ transversally along a conic bundle. Note that $V_0$ lies over the double line $L$, and thus appears with multiplicity two.
 
\end{enumerate}

In suitable coordinates, $\mc{X}^{(4)}$ can be written as $FG^2=p$, where $F$ stands for local equations of the rigid Calabi-Yau threefold $R_0$ and $G$ stands for local equations of the $\mathbb{P}^2$-bundle $V_0$. Let $\pi$ be a uniformizer of $\mathbb{Q}_p[\sqrt{p}]$. To obtain reduced components in the degenerate fiber, first we take a pullback of the family by $p \mapsto \pi^2$, obtaining $\mc{X}^{(5)}:=\mc{X}^{(4)}\times_{\mathbb{Z}_p}\operatorname{Spec}\mathbb{Z}_p[\sqrt{p}]$.

We can write the equation of $\mathcal{X}^{(5)}$ in the form $FG^2 = \pi^2$. Take a double cover of $\mathcal{X}^{(5)}$ branched along $F=0$. It embeds into the zero locus of $FG^2 = \pi^2,\ u^2=F$. The latter is reducible, so consider one of its irreducible components; call it $\mc{Y}$. Its equations are $uG = \pi, u^2=F$, and there is a morphism $\mc{Y} \rr \mc{X}^{(4)}$ sending $(F,G,u) \mapsto (F',G')=(u,G)$, where $F',G'$ denote the local coordinates on $\mc{X}^{(4)}$.

The scheme $\mc{Y}$ is a semi-stable degeneration of $\mc{X}$. Its generic fiber $Y_\infty$ is isomorphic to the generic fiber of $\mathcal{X}^{(4)}$, while the special fiber $Y_0$ has two components: a smooth rigid Calabi-Yau threefold $R_0$ and a quadric bundle $Q_0 \rightarrow L$ which replaced the $\mathbb{P}^2$-bundle $V_0\rightarrow L$. The bundle $Q_0$ has four singular fibers over four pinch-points.

This construction yields the following result:
\begin{theorem}\label{th:con}
There exists a scheme $\mathcal{Y}\rightarrow\operatorname{Spec}\mathbb{Z}_p[\sqrt{p}]$ such that the generic fiber $Y_\infty\rightarrow\operatorname{Spec}\mathbb{Q}_p[\sqrt{p}]$ is isomorphic to $X_\infty$ and the special fiber $Y_0\rightarrow\operatorname{Spec}\mathbb{F}_p$ is a union of a rigid Calabi-Yau threefold $R_0$ and a quadric bundle $Q_0$, intersecting transversally in a conic bundle.
\end{theorem}

The irreducible components $R_0$ and $Q_0$ have natural lifts to characteristic zero (see \cite{CYK,CvS}); we denote them by $R_\infty$ and $Q_\infty$. A key observation in the up-coming proof that the Galois action on $H^3_{\acute{e}t}(Y_\infty,\mathbb{Q}_p)$ is crystalline is the following identification:

\begin{theorem}\label{th:coh}
There is an isomorphism
$$H^3_{\acute{e}t}(Q_\infty,\mathbb{Q}_p) \simeq H^1_{\acute{e}t}(E,\mathbb{Q}_p)(-1),$$
where $E$ is an elliptic curve with $j$-invariant $1728$.
\end{theorem}

\begin{proof}

Let $Q:=Q_\infty$ and let $\QQ_Q$, resp. $\QQ_E$, denote the locally constant sheaf $\QQ_p$ on $Q$, resp. on $E$. The second page of the Leray spectral sequence for the morphism $\pi \colon Q \rr L$, $L\simeq\mathbb{P}^1$, is:
\begin{center}
\begin{tikzcd}[row sep=small]

H^0(\PP^1 , R^4\pi_*\QQ_Q) \arrow[rrd]& H^1(\PP^1 , R^4\pi_*\QQ_Q) & H^2(\PP^1 , R^4\pi_*\QQ_Q) \\
H^0(\PP^1 , R^3\pi_*\QQ_Q) \arrow[rrd]& H^1(\PP^1 , R^3\pi_*\QQ_Q) & H^2(\PP^1 , R^3\pi_*\QQ_Q) \\
H^0(\PP^1 , R^2\pi_*\QQ_Q) \arrow[rrd]& H^1(\PP^1 , R^2\pi_*\QQ_Q) & H^2(\PP^1 , R^2\pi_*\QQ_Q) \\
H^0(\PP^1 , R^1\pi_*\QQ_Q) \arrow[rrd]& H^1(\PP^1 , R^1\pi_*\QQ_Q) & H^2(\PP^1 , R^1\pi_*\QQ_Q) \\
H^0(\PP^1 , R^0\pi_*\QQ_Q) & H^1(\PP^1 , R^0\pi_*\QQ_Q) & H^2(\PP^1 , R^0\pi_*\QQ_Q)

\end{tikzcd}
\end{center}
The sheaves $R^1\pi_*\QQ_Q$ and $R^3\pi_*\QQ_Q$ are zero sheaves and consequently the appropriate rows in the spectral sequence vanish. This sequence degenerates on the second page; hence $H^1(L, R^2\pi_*\QQ_Q) \simeq H^3_{\acute{e}t}(Q,\mathbb{Q}_p)$. We need to show that $H^1(L, R^2\pi_*\QQ_Q) \simeq H^1_{\acute{e}t}(E,\mathbb{Q}_p)(-1)$, where elliptic curve $E$ is the double cover of $L$ branched along four pinch points of $Q$. We denote the covering map by $\tau:E \rr L$.

Sections of the sheaf $R^2\pi_*\QQ_Q$ are generated by rulings on the fibers of $Q$. Let $P \in Q$ be a point which does not lie on a singularity of a fiber, and let $l(P)$ denote the lines through $P$ contained in the fiber $\pi^{-1}(\pi(P))$. When $P$ lies on a smooth fiber, there are two such lines. When $P$ lies on a singular fiber, there is one such line with multiplicity two. 

As a bundle of quadric surfaces over $L$, $Q$ is given by a degree two polynomial over the field of rational functions on $L$. By Tsen's theorem (see \cite{Tsen}) there exists a section $s \colon L \rr Q$. Let $V := \left\{\left(l(s(P)),P\right) \colon P \in L\right\}$. The variety $V$ is a bundle of singular conics over $L$. Its generic fiber consists of two intersecting lines, and its degenerate fibers are double lines. Since the degenerate fibers of $V$ lie over the same points of $L$ as the degenerate fibers of $Q$, the bundles have the same determinants. Moreover, $R^2\pi_*\QQ_Q = R^2\pi_*\QQ_V$.

Consider the fiber product $V \times_{L} E$. The map $\tau:E\rightarrow L$ is a double cover branched over the vanishing locus of the determinant of $V$. Since the generic fiber of $V \times_{L} E\rightarrow L$ consists of two identical pairs of intersecting lines, the determinant of the pullback $\tau^*V = V \times_{L} E$ is a square and $V \times_{L} E$ is reducible. Let $S$ be one of its irreducible components.

Since $E$ is irreducible, the projection $\alpha \colon S \rr E$ is surjective and the fiber of $S\rightarrow L$ over a generic point consists of two lines. The other projection $\beta \colon S \rr V$ is also surjective, hence $R^2\pi_*\QQ_V = R^2(\pi \beta)_*\QQ_S = R^2(\tau \alpha)_*\QQ_S = \tau_* \QQ_E \otimes H^2(\PP^1) = \tau_* \QQ_E(-1)$.
\end{proof}

\section{Crystalline cohomology}\label{sec:crystalline}

Let us fix some notation for the rest of the paper. $K$ is a $p$-adic field containing $\sqrt{p}$, $\mathcal{O}_K$ is its ring of integers, $\mathbb{F}_q$ is its residue field. For a scheme $\mathcal{A}\rightarrow\operatorname{Spec}\mathcal{O}_K$, we denote by $A_\infty\rightarrow\operatorname{Spec}K$ its generic fiber and by $A_0\rightarrow \operatorname{Spec}\mathbb{F}_q$ its special fiber. $\mathcal{X}$ is the family of double octics defined by (\ref{double-octic}); $\mathcal{Y}$ is as in Theorem \ref{th:con}; $\mathcal{R}$ is the rigid Calabi-Yau threefold and $\mathcal{Q}$ is the quadric bundle, both described in Theorem \ref{th:con} and paragraphs below it.

In this section, we prove the first part of our Main Theorem: the Galois representation
\[
G_{K}\longrightarrow \operatorname{Aut}\left( H^3_{\acute{e}t}(Y_\infty,\mathbb{Q}_p )\right)
\]
is crystalline. In particular, the action on the $\ell$-adic cohomology for $\ell\neq p$ is unramified.

For the sake of completeness let us state the definitions.

\begin{definition}
Let $K$ be a $p$-adic field with residue field $k$, $G_K$ and $G_k$ their absolute Galois groups and $I_K$ the inertia subgroup, i.e. the kernel of the natural homomorphism $G_K\rightarrow G_k$. Let $V$ be a $\mathbb{Q}_\ell$-adic vector space.

A representation $\rho: G_K\rightarrow \mathrm{GL}(V)$ is called \emph{unramified} if $\rho(I_K)=\left\{\operatorname{Id}_V\right\}$.

Assume $\ell=p$. A representation $\rho: G_K\rightarrow \mathrm{GL}(V)$ is called \emph{crystalline} if the equality \mbox{$\dim_{K_0}\left(B_{cris}\otimes_{\mathbb{Q}_p}V\right)^{G_K}=\dim_{\mathbb{Q}_p}V$} holds, where $B_{cris}$ is the Fontaine-Illusie-Kato crystalline module and $K_0$ is the field of fractions of Witt vectors.
\end{definition}

In our proof we follow the strategy from \cite{CvS}, adapting complex geometry arguments to the $p$-adic setting. For example, existence of the limiting mixed Hodge structure at a singular point of a Picard-Fuchs operator corresponds to the existence of \emph{logarithmic structure} $H^j_{log-cryst}(Y_0/W)$ on $Y_0$. The exact definition of this structure can be found in \cite{Kato}.

Let $S:=(\operatorname{Spec}(\mathbb{F}_q),\mathbb{N}\oplus\mathbb{F}_q^*)$, let $W=W(\mathbb{F}_q)$ be the Witt ring and $K_0$ its field of fractions. Let $Z_0/S$ be a proper SNCL variety of pure dimension, and finally let $Z_0^{(j)}$ denote the disjoint union of $j$-fold intersections of the irreducible components of $Z_0$.

Proof of our Main Theorem relies on the \emph{$p$-adic Clemens-Schmidt sequence}:
\begin{lemma}[Theorem 3.6, \cite{Nakka}]\label{l:nakka}
In the setup as above, the spectral sequence
$$
E^{-k,h+k}_1=\bigoplus_{j\geq \max\{-k,0\}}H^{h-2j-k}_{cryst}(Z_0^{(2j+k+1)}/W)(-j-k)
$$
$$
\Longrightarrow H^h_{log-cryst}(Z_0/W)
$$
degenerates at $E_2$ modulo torsions.
\end{lemma}
Lemma \ref{l:nakka} allows us to compute $H^h_{log-cryst}(Z_0/W)$ but we need to compute $H^3_{\acute{e}t}(Z_\infty,\mathbb{Q}_p)$. For this purpose we use a semi-stable comparison theorem between \'etale cohomology and logarithmic structure. It works with coefficients in the ring of semi-stable periods $B_{st}$. This ring contains $B_{cris}$ and $B_{cris}=B_{st}^{N=0}$, where $N$ is the monodromy operator. In general, we have inequalities $\dim_{K_0}\left(B_{cris}\otimes_{\mathbb{Q}_p}V\right)^{G_K}\leq \dim_{K_0}\left(B_{st}\otimes_{\mathbb{Q}_p}V\right)^{G_K}\leq\dim_{\mathbb{Q}_p}V$.

\begin{lemma}[Theorem 0.2, \cite{Tsuji}]\label{l:tsuji}
In the setup as above, there is an isomorphism
$$
B_{st}\otimes_{\mathbb{Q}_p}H_{\acute{e} t}(Z_\infty,\mathbb{Q}_p)\simeq B_{st}\otimes_{W}H_{log-cryst}(Z_0/W)
$$
compatible with all the natural structures.
\end{lemma}

Using Lemmas \ref{l:nakka} and \ref{l:tsuji}, as well as the geometric description from Section \ref{sec:construction}, we can restate the calculations from Section 3.2 of \cite{CvS} in mixed characteristic.

\begin{proposition}
There is a direct sum decomposition
$$
B_{st}\otimes_{\mathbb{Q}_p}H_{\acute{e} t}^3(Y_\infty,\mathbb{Q}_p)\simeq B_{st}\otimes_{\mathbb{Q}_p}\left(H_{\acute{e} t}^3(R_\infty,\mathbb{Q}_p)\oplus H_{\acute{e} t}^1(E,\mathbb{Q}_p)(-1)\right),
$$
where $R_\infty$ is a rigid Calabi-Yau threefold and $E$ is an elliptic curve.
\end{proposition}
\begin{proof}
Recall that $Y_\infty$ is the generic fiber of the projective, semi-stable scheme $\mathcal{Y}\rightarrow\operatorname{Spec}\mathcal{O}_K$ with central fiber $Y_0$. In particular, $Y_0$ is a proper SNCL variety. In the notation of Lemma \ref{l:tsuji}, we have $Y^{(1)}_0=R_0\sqcup Q_0$, $C_0:=Y^{(2)}_0=R_0\cap Q_0$ is a conic bundle, $Y_0^{(j)}=\varnothing$ for $j\geq 3$.

The natural lifts $\mathcal{R},\mathcal{Q}$ and $\mathcal{C}$ are all smooth. Three applications of Lemma \ref{l:tsuji} show that the first page of the $p$-adic Clemens-Schmidt spectral sequence reads:
\begin{center}
\begin{tikzcd}
B_{st}\otimes H^4_{\acute{e}t}(C_\infty,\mathbb{Q}_p)(-1) \arrow[r] & B_{st}\otimes H_{\acute{e}t}^6(R_\infty,\mathbb{Q}_p)\oplus B_{st}\otimes H_{\acute{e}t}^6(Q_\infty,\mathbb{Q}_p) & \ \\ 
B_{st}\otimes H^3_{\acute{e} t}(C_\infty,\mathbb{Q}_p)(-1) \arrow[r] & B_{st}\otimes H_{\acute{e} t}^5(R_\infty,\mathbb{Q}_p)\oplus B_{st}\otimes H_{\acute{e} t}^5(Q_\infty,\mathbb{Q}_p) & \ \\ 
B_{st}\otimes H^2_{\acute{e}t}(C_\infty,\mathbb{Q}_p)(-1) \arrow[r] & B_{st}\otimes H_{\acute{e}t}^4(R_\infty,\mathbb{Q}_p)\oplus B_{st}\otimes H_{\acute{e}t}^4(Q_\infty,\mathbb{Q}_p) \arrow[r] & B_{st}\otimes H^4_{\acute{e}t}(C_\infty,\mathbb{Q}_p)\ \\
B_{st}\otimes H^1_{\acute{e}t}(C_\infty,\mathbb{Q}_p)(-1) \arrow[r] & B_{st}\otimes H_{\acute{e}t}^3(R_\infty,\mathbb{Q}_p)\oplus B_{st}\otimes H_{\acute{e}t}^3(Q_\infty,\mathbb{Q}_p) \arrow[r] & B_{st}\otimes H^3_{\acute{e}t}(C_\infty,\mathbb{Q}_p) \\
B_{st}\otimes H^0_{\acute{e}t}(C_\infty,\mathbb{Q}_p)(-1) \arrow[r] & B_{st}\otimes H_{\acute{e}t}^2(R_\infty,\mathbb{Q}_p)\oplus B_{st}\otimes H_{\acute{e}t}^2(Q_\infty,\mathbb{Q}_p) \arrow[r] & B_{st}\otimes H^2_{\acute{e}t}(C_\infty,\mathbb{Q}_p) \\
\ & B_{st}\otimes H_{\acute{e}t}^1(R_\infty,\mathbb{Q}_p)\oplus B_{st}\otimes H_{\acute{e}t}^1(Q_\infty,\mathbb{Q}_p)\arrow[r] \arrow[r] & B_{st}\otimes H^1_{\acute{e}t}(C_\infty,\mathbb{Q}_p)\ \\
\ & B_{st}\otimes H_{\acute{e}t}^0(R_\infty,\mathbb{Q}_p)\oplus B_{st}\otimes H_{\acute{e}t}^0(Q_\infty,\mathbb{Q}_p)\arrow[r]\arrow[r] & B_{st}\otimes H^0_{\acute{e}t}(C_\infty,\mathbb{Q}_p)\\
\end{tikzcd}
\end{center}
By Lemma \ref{l:nakka} the spectral sequence degenerates at the second page and converges to $B_{st}\otimes H^h_{log-cryst}(Y_0/W)$. Thus the latter has $B_{st}\otimes\left(H_{\acute{e}t}^3(R_\infty,\mathbb{Q}_p)\oplus H^3_{\acute{e}t}(Q_\infty,\mathbb{Q}_p)\right)$ as a direct summand. By Theorem \ref{th:coh} there is an elliptic curve $E$ such that $H^3_{\acute{e}t}(Q_\infty,\mathbb{Q}_p)\simeq H^1_{\acute{e}t}(E,\mathbb{Q}_p)(-1)$. In particular, $b_3(Q_\infty)=b_1(E)=2$. We can check equality $b_3(Y_\infty)=4$ in characteristic zero, where it is part of the main Theorem of \cite{CvS}. Since $R_\infty$ is rigid, $b_3(R_\infty)=2$. Lemma \ref{l:tsuji} and dimension count imply that $B_{st}\otimes H^h_{log-cryst}(Y_0/W)\simeq B_{st}\otimes\left(H_{\acute{e}t}^3(R_\infty,\mathbb{Q}_p)\oplus H^3_{\acute{e}t}(Q_\infty,\mathbb{Q}_p)\right)$. Using Lemma \ref{l:tsuji} again gives the claim.
\end{proof}

Now we prove the first part of the Main Theorem.

\begin{theorem}\label{th:cryst}
The Galois action on the cohomology group $H^3_{\acute{e}t}(Y_\infty,\mathbb{Q}_p)$ is crystalline.
\end{theorem}

\begin{proof}
Put $H_Y:=H^3_{\acute{e}t}(Y_\infty,\mathbb{Q}_p)$, $H_R:=H^3_{\acute{e}t}(R_\infty,\mathbb{Q}_p)$ and $H_E:=H^1_{\acute{e}t}(E,\mathbb{Q}_p)(-1)$. By the previous proposition there is an isomorphism $B_{st}\otimes H_Y\simeq B_{st}\otimes (H_R\oplus H_E)$. Using the fact that this decomposition respects the monodromy action, we get $B_{cris}\otimes H_Y=(B_{st}\otimes H_Y)^{N=0}\simeq(B_{st}\otimes (H_R\oplus H_E))^{N=0}=(B_{st}\otimes H_R)^{N=0}\oplus (B_{st}\otimes H_E)^{N=0}=(B_{cris}\otimes H_R)\oplus (B_{cris}\otimes H_E)$. $R_\infty$ and $E$ have smooth models, so Galois representations $H_R$ and $H_E$ are crystalline. Thus $\dim_{K_0}\left(B_{cris}\otimes H_R\right)^{G_K}=\dim_{K_0}\left(B_{cris}\otimes H_E\right)^{G_K}=2$ and consequently
\begin{eqnarray*}
\dim_{K_0}\left(B_{cris}\otimes H_Y\right)^{G_K}=\dim_{K_0}\left(B_{cris}\otimes H_R\right)^{G_K}+\dim_{K_0}\left(B_{cris}\otimes H_E\right)^{G_K}=4=b_3(Y_\infty)=\dim_{\mathbb{Q}_p}(H_Y)    
\end{eqnarray*}
\end{proof}

\section{Lack of a smooth model}\label{sec:no-good}

Given Theorem \ref{th:cryst}, showing that $Y_\infty$ does not admit potentially good reduction will conclude the proof of our Main Theorem.

In characteristic zero, bad reduction means that the family over the punctured unit disc cannot be completed to a smooth family. To prove it one can use general theory of degenerations of Calabi-Yau threefolds (see \cite{CvS}). In mixed characteristic we must take a different approach. Let us recall the definitions.

\begin{definition}[Definition 2.1, \cite{BLL}]

\ 

A variety $V$ defined over a $p$-adic field $K$ has \emph{potentially good reduction} if there exists a finite field extension $L/K$ with the ring of integers $\mathcal{O}_L$ and a smooth proper algebraic space \mbox{$\mathcal{V}\rightarrow\operatorname{Spec}\mathcal{O}_L$} flat over $\mathcal{O}_L$ such that $V_\infty$ is isomorphic to $V$ over $L$ and the special fiber $V_0$ is a scheme.

A variety has \emph{bad reduction} if it does not have potentially good reduction.
\end{definition}
\noindent Note that we allow smooth models over fields larger than the field of definition. Thus our notion of bad reduction means \emph{bad reduction over} $\overline{K}$, not only over $K$. Furthermore, we work in the category of \emph{algebraic spaces} rather than schemes, since in the latter the N\'eron-Ogg-Shafarevich criterion can fail even for $K3$ surfaces (see \cite{Matsumoto}).

In the proof we use several general lemmas. Recall that a variety is \emph{ruled} if it is birational to the product $\mathbb{P}^1\times V$ for some variety $V$.

\begin{lemma}\label{lem:ruled}
If $A$ is a Calabi-Yau variety, then it is not ruled.
\end{lemma}

\begin{proof}
The geometric genus $p_g(V):=h^{n,0}(V)$ of a variety is a birational invariant (Theorem 8.19, \cite{Hart}). Since we have $p_g(\mathbb{P}^1\times V)=0$ and $p_g(A)=1$, the claim follows.
\end{proof}

\begin{lemma}[Proposition 3, \cite{Ab}]\label{lem:ab}
Let $f:X\rightarrow Y$ be a proper birational morphism with $Y$ regular. Then every irreducible component of the exceptional locus of $f$ is ruled.
\end{lemma}

\begin{lemma}[Corollary 5.7.14, \cite{RG}]\label{lem:chow}
Let $Y$ be a coherent algebraic space, $X \rr Y$ a separated $Y$-algebraic space of finite type and $U\subset X$ an open subspace quasi-projective over $Y$. Then there exists an $U$-admissible blow-up $\pi \colon \wdt{X} \rr X$ such that $\wdt{X}$ is quasi-projective over $Y$.
\end{lemma}

\begin{lemma}\label{lem:batyrev}
Let $f:A\rightarrow B$ be a birational map of Calabi-Yau threefolds defined over a finite field. There exist open subsets $U\subset A$ and $V\subset B$ such that $\textnormal{codim}_A (A\setminus U),\ \textnormal{codim}_B (B\setminus V)\geq 2$ and $f\big| _U:U\rightarrow V$ is an isomorphism.
\end{lemma}

\begin{proof}
 Proof is the same as that of Proposition 3.1 in \cite{batyrev_1999} for the complex case.
\end{proof}

\begin{lemma}\label{lem:iso-bi}
Let $\mc{A}$ and $\mc{B}$ be projective, normal, irreducible schemes over $\textnormal{Spec}\ \mc{O}_K$ such that the generic fibers $A_\infty,B_\infty$ are isomorphic. Moreover, let us assume that $\mc{B}$ is regular. Then every non-ruled component of the special fiber $A_0$ is birational to a component of $B_0$.
\end{lemma}

\begin{proof}
Any isomorphism $f:A_\infty\rightarrow B_\infty$ induces a birational map $\overline{f}:\mathcal{A}\rightarrow\mathcal{B}$ of schemes over $\operatorname{Spec}\mathcal{O}_K$. Consider the normalization of the (closure of the) graph $\Gamma_f$ of $\overline{f}$:
$$
\mathcal{Z}:={\overline{\Gamma_f}}^\nu
$$
where $\nu:=\nu_V:V^\nu\rightarrow V$ is the the normalization morphism as in Definition 29.54.1 of \cite{stacks}. Since $\overline{\Gamma_f}\subset \mathcal{A}\times \mathcal{B}$, we have natural projections $\pi_A:\overline{\Gamma_f}\rightarrow \mathcal{A}$, $\pi_B:\overline{\Gamma_f}\rightarrow \mathcal{B}$. Consider the exceptional locus $E_A\subset \mathcal{Z}$, resp. $E_B$, of the birational morphism \mbox{$f_A:=\pi_A\circ\nu_{\overline{\Gamma_f}}:\mathcal{Z}\rightarrow \mathcal{A}$}, resp. \mbox{$f_B:=\pi_B\circ\nu_{\overline{\Gamma_f}}:\mathcal{Z}\rightarrow \mathcal{B}$}.

Let $A_0'$ be a non-ruled component of $A_0$. Since $E_A=\{z\in \mathcal{Z}: \dim \left(f^{-1}(f(z))>0\right)\}$ and $\operatorname{codim}_\mathcal{A} A_0'=1$, it follows that $A_0'\setminus f_A(E_A)\neq\varnothing$. As $f_A$ is an isomorphism away from $E_A$, we conclude that $A_0'$ and $Z_0':=f_A^{-1}\left(A'_0\setminus f_A(E_A)\right)$ are birational.

In particular, $Z_0'$ is not ruled. However, by Lemma \ref{lem:ab} every irreducible component of the exceptional locus $E_B$ is ruled, hence $Z'_0\not\subset E_B$ by Zariski's Main Theorem \cite{Zariski}. It follows that $Z_0'\setminus E_B$ is birational to some irreducible component of $B_0$, which concludes the proof.
\end{proof}

Now we are ready to prove the remaining part of the Main Theorem.

\begin{theorem}\label{th:bad}
Calabi-Yau threefold $Y_\infty$ has bad reduction.
\end{theorem}

\begin{proof}
Assume that $Y_\infty$ has good reduction, w.l.o.g over $K$. Let $\mc{C}\rightarrow\operatorname{Spec}\mathcal{O}_K$ be a smooth algebraic space such that $Y_\infty\simeq C_\infty$. The special fiber $Y_0$ has two irreducible components $R_0$ and $Q_0$, as in Theorem \ref{th:con}, while $C_0$ is a smooth Calabi-Yau threefold defined over a finite field such that $b_3(C_0)=4$.

By Lemma \ref{lem:chow} there exists a blow-up $\pi:\mc{D}\rightarrow\mc{C}$ at a closed subscheme of $C_0$ such that $\mc{D}$ is projective over $\text{Spec} \, \mc{O}_L$ and the generic fiber $D_\infty$ is isomorphic to $C_\infty \simeq Y_\infty$. The special fiber $D_0$ contains a component birational to $C_0$.

Let $\mc{E}$ be the normalization of $\mc{D}$. Schemes $\mc{E}$ and $\mc{Y}$ are projective, irreducible and normal. $R_0$ is a Calabi-Yau threefold, and $E_0$ contains a (smooth) Calabi-Yau component $E_0'$; both are not ruled by Lemma \ref{lem:ruled}. It follows from Lemma \ref{lem:iso-bi} that the special fibers' components $E_0'$ and $R_0$ are birational.

By Lemma \ref{lem:batyrev} there exist closed subsets $F\subset E_0',\ G\subset R_0$ of dimension at most $1$ such that $E_0'\setminus F$ and $R_0\setminus G$ are isomorphic. In particular, we have the equality of zeta functions:
\begin{equation}\label{eq:Z}
Z(E_0',s)\cdot Z(G,s)=Z(R_0,s)\cdot Z(F,s)
\end{equation}
For any rational function $f\in\mathbb{Q}(T)$ we define the multi-sets (counted with multiplicities):
$$
\mathcal{Z}(f):=f^{-1}(\{0\}),\quad \mathcal{Z}_i(f):=\mathcal{Z}(f)\cap\{ x\in\overline{\mathbb{Q}}: \forall_{\iota:\overline{\mathbb{Q}}\hookrightarrow\mathbb{C} }\ |\iota(x)|=q^{-\tfrac{i}{2}}\}
$$
By the Weil conjectures, functions $Z(E_0',s),\ Z(R_0,s)$ are rational (see \cite{Dwork}). By the Weil conjectures for (possibly singular) curves, so are $Z(F,s)$ and $Z(G,s)$ (see \cite{AP}). Moreover, all roots of $Z(F,s)$ and $Z(G,s)$ have absolute value $1$ or $\sqrt{q}^{-1}$. Hence,
$$\mathcal{Z}\left(Z(F,s)\right)=\mathcal{Z}_{0}\left(Z(F,s)\right)\cup \mathcal{Z}_{1}\left( Z(F,s)\right)\ \textnormal{and}\ \mathcal{Z}\left(Z(G,s)\right)=\mathcal{Z}_{0}\left(Z(G,s)\right)\cup \mathcal{Z}_{1}\left( Z(G,s)\right)$$
Combining this with (\ref{eq:Z}), we obtain $$\mathcal{Z}_3\Big(Z(E_0',s)\Big)=\mathcal{Z}_3\Big(Z(E_0',s)\cdot Z(F,s)\Big)=\mathcal{Z}_3\Big(Z(R_0,s)\cdot Z(G,s)\Big)=\mathcal{Z}_3\Big(Z(R_0,s)\Big)$$
However, by the Riemann hypothesis (see \cite{Deligne}) we have $|\mathcal{Z}_3(Z(E_0',s))|=b_3(E_\infty)=b_3(Y_\infty)=4$ and $|\mathcal{Z}_3(Z(R_0,s))|=b_3(R_\infty)=2$, 
a contradiction.
\end{proof}

\end{document}